\documentclass[a4paper]{amsart}

\usepackage{geometry}
\usepackage{verbatim}
\usepackage{enumitem}

\usepackage{amssymb}
\usepackage[mathscr]{euscript}
\usepackage{MnSymbol}
\usepackage{wasysym}
\usepackage{todonotes}
\usepackage{stmaryrd}
\usepackage{tikz-cd}
\usepackage{nicefrac}
\tikzcdset{arrow style=math font}

\usepackage[bbgreekl]{mathbbol}
\DeclareSymbolFontAlphabet{\mathbb}{AMSb}
\DeclareSymbolFontAlphabet{\mathbbl}{bbold}

\DeclareMathOperator{\map}{map}

\DeclareMathOperator{\K}{K}
\DeclareMathOperator{\THH}{THH}
\DeclareMathOperator{\TC}{TC}

\DeclareMathOperator{\TR}{TR}

\DeclareMathOperator{\LMod}{LMod}

\DeclareMathOperator{\Perf}{Perf}

\DeclareMathOperator{\Proj}{Proj}

\newcommand{\Sp}{\mathrm{Sp}}

\newcommand{\Alg}{\mathrm{Alg}}

\newcommand{\Mod}{\mathrm{Mod}}

\newcommand{\CycSp}{\mathrm{CycSp}}

\renewcommand{\S}{\mathbf{S}}

\newcommand{\Z}{\mathbb{Z}}
\newcommand{\F}{\mathbb{F}}

\newcommand{\E}{\mathbb{E}}

\newcommand{\cat}[1]{\mathscr{#1}}

\usepackage{hyperref}

\newtheorem{theorem}{Theorem}[section]
\newtheorem{proposition}[theorem]{Proposition}

\newtheorem{corollary}[theorem]{Corollary}

\newtheorem{theoremvar}{Theorem}
\newtheorem{corollaryvar}{Corollary}
\newenvironment{theorem*}[1]{\renewcommand{\thetheoremvar}{#1}\theoremvar}{\endtheoremvar}
\newenvironment{corollary*}[1]{\renewcommand{\thecorollaryvar}{#1}\corollaryvar}{\endcorollaryvar}

\theoremstyle{definition}
\newtheorem{definition}[theorem]{Definition}

\newtheorem{remark}[theorem]{Remark}
\newtheorem{example}[theorem]{Example}

\title{A chromatic vanishing result for TR}

\author{Liam Keenan}
\address{Department of Mathematics, University of Minnesota, USA}
\email{keena169@umn.edu}

\author{Jonas McCandless}
\address{Max Planck Institute for Mathematics, Bonn, Germany}
\email{mccandless@mpim-bonn.mpg.de}

\begin{document}

\begin{abstract}
In this note, we establish a vanishing result for telescopically localized TR. More precisely, we prove that $T(k)$-local TR vanishes on connective $L_n^{p,f}$-acyclic $\E_1$-rings for every $1 \leq k \leq n$ and deduce consequences for connective Morava K-theory and the Thom spectra $y(n)$. The proof relies on the relationship between TR and the spectrum of curves on K-theory together with fact that algebraic K-theory preserves infinite products of additive $\infty$-categories which was recently established by C{\'o}rdova Fedeli. 
\end{abstract}

\maketitle

\section{Introduction} \label{section:introduction}
In this note, we study the telescopic localizations of TR inspired by the work of Land--Mathew--Meier--Tamme~\cite{LMT20} and Mathew~\cite{Mat20}. Our starting point is the following result which follows from the main result of~\cite{LMT20}: If $R$ is an $\E_1$-ring with $L_n^{p,f} R \simeq 0$, then
\[
	L_{T(k)}\K(R) \simeq 0
\]
for every $1 \leq k \leq n$. For instance, if $R = \Z/p^n$ for some integer $n \geq 1$, then $L_{T(1)}\K(\Z/p^n) \simeq 0$. We consider this result as an extension of Quillen's fundamental calculation that $\K(\F_p)_p^{\wedge} \simeq H\Z_p$ which in particular yields that $L_{T(1)}\K(\F_p) \simeq 0$. This particular consequence was also obtained by Bhatt--Clausen--Mathew~\cite{BCM20} by means of a calculation in prismatic cohomology. Additionally, the vanishing result above for $T(k)$-local K-theory can be applied to the Morava K-theories $K(n)$ and to the Thom spectra $y(n)$ considered by Mahowald--Ravenel--Shick in~\cite{MRS01}.

\subsection{Results}
We will be interested in similar vanishing results for $T(k)$-local $\TR$\footnote{Note that $L_{T(k)}\TR(R) \simeq L_{T(k)}\TR(R, p)$, where $\TR(R, p)$ denotes the $p$-typical version of $\TR$. Indeed, the canonical map $\TR(R) \to \TR(R, p)$ is a $p$-adic equivalence and $T(n)$-localization is insensitive to $p$-completion. Therefore, we will not distinguish between $p$-typical TR and integral TR in this note.}. The invariant $\TR$ plays an instrumental role in the classical construction of topological cyclic homology in~\cite{BHM93,HM97,BM16}, where $\TC$ is obtained as the fixedpoints of a Frobenius operator on $\TR$. In~\textsection\ref{section:chromatic_vanishing_results}, we briefly review the construction of $\TR$ following~\cite{McC21} which produces $\TR$ together with its Frobenius operator entirely in the Borel--equivariant formalism of Nikolaus--Scholze~\cite{NS18}. Even though $\TR$ does not feature prominently in the construction of $\TC$ given in~\cite{NS18}, $\TR$ remains an important invariant by virtue of its close relationship to the Witt vectors and the de Rham--Witt complex~\cite{Hes96,HM97,HM03,HM04}. In~\cite{Mat20}, Mathew proves that $T(1)$-local $\TR$ is truncating on connective $H\Z$-algebras which means that if $R$ is a connective $H\Z$-algebra, then the canonical map of spectra
\[
	L_{T(1)} \TR(R) \to L_{T(1)} \TR(\pi_0 R)
\]
is an equivalence. This property was verified for $T(1)$-local K-theory and $T(1)$-local $\TC$ in~\cite{BCM20,LMT20}. Our main result is a version of this at higher chromatic heights:

\begin{theorem*}{A} \label{theorem:A}
Let $n \geq 1$. If $R$ is a connective $\E_1$-ring such that $L_n^{p,f} R \simeq 0$, then
\[
	L_{T(k)} \TR(R) \simeq 0
\] 
for every $1 \leq k \leq n$.
\end{theorem*}

We remark that Theorem~\ref{theorem:A} is a consequence of the work of~\cite{LMT20} in the case where $R$ admits a more refined multiplicative structure; If $R$ admits an $\E_m$-ring structure for $m \geq 2$, then the refined cyclotomic trace $\K(R) \to \TR(R)$ is a map of $\E_1$-rings. Consequently, the spectrum $L_{T(k)}\TR(R)$ admits the structure of a $L_{T(k)}\K(R)$-module and $L_{T(k)}\K(R) \simeq 0$ by~\cite[Theorem 3.8]{LMT20}. A similar sort of reasoning has recently been employed with great success to study redshift phenomena for algebraic K-theory in~\cite{BSY22,CMNN20b,HW20,Yuan21}. We deduce the following results from Theorem~\ref{theorem:A}: 

\begin{corollary*}{B} \label{corollary:B}
Let $n \geq 1$. Then $L_{T(k)}\TR(\Z/p^n) \simeq 0$ for every $k \geq 1$. 
\end{corollary*}

We stress that Corollary~\ref{corollary:B} is a consequence of the work of~\cite{BCM20,LMT20} by the reasoning above. For $n = 1$, Corollary~\ref{corollary:B} can also be deduced from the work of Mathew~\cite{Mat20}. Since $T(1)$-local $\TR$ is truncating on connective $H\Z$-algebra it is in particular nilinvariant by~\cite{LT19}, so
\[
	L_{T(1)}\TR(\Z/p^n) \simeq L_{T(1)}\TR(\F_p) \simeq 0,
\]
where the final equivalence follows since $\TR(\F_p, p) \simeq H\Z_p$ by Hesselholt--Madsen~\cite{HM97}. As a consquence of Theorem~\ref{theorem:A} we deduce a new chromatic vanishing result for the connective Morava K-theories, which we denote by $k(n)$.
While $k(n)$ admits the structure of an $\E_1$-ring, it does not admit the structure of an $\E_2$-ring so we cannot argue using the refined cyclotomic trace above.

\begin{corollary*}{C} \label{corollary:C}
Let $n \geq 2$. Then $L_{T(k)}\TR(k(n)) \simeq 0$ for every $1 \leq k \leq n-1$. 
\end{corollary*}

Similarly, we obtain a chromatic vanishing result for the Thom spectra $y(n)$ considered in~\cite{MRS01}.

\subsection{Methods}
We end by explaining the strategy of our proof of Theorem~\ref{theorem:A}. They key input is the close relationship between TR and the spectrum of curves on K-theory as studied in~\cite{BS05,Blo77,Hes96,McC21}. For every $\E_1$-ring $R$, the spectrum of curves on K-theory is defined by
\[
	\mathrm{C}(R) = \varprojlim_i \Omega \tilde{\K}(R[t]/t^i),
\]
where $\tilde{\K}(R[t]/t^i)$ denotes the fiber of the map $\K(R[t]/t^i) \to \K(R)$ induced by the augmentation. If we assume that $R$ is connective, then $\TR(R) \simeq \mathrm{C}(R)$ by~\cite[Corollary 4.2.5]{McC21}. This result was preceded by Hesselholt~\cite{Hes96} and Betley--Schlichtkrull~\cite{BS05} who established the result for associative rings after profinite completion. Combining the theorem of the weighted heart (cf.~\cite{Fon18,Hel20,HS21}) with the recent result of C{\'o}rdova Fedeli~\cite[Corollary 2.11.1]{Cor23} which asserts that algebraic K-theory preserves arbitrary products of additive $\infty$-categories, we reduce to proving that
\[
	L_{T(k)} \K^{\oplus}\big( \prod_{i \geq 1} \Proj^{\omega}_{R[t]/t^i} \big) \simeq 0
\]
provided that $L_n^{p,f}R \simeq 0$, where $\Proj^{\omega}_{R[t]/t^i}$ denotes the additive $\infty$-category of finitely generated projective $R[t]/t^i$-modules and $\K^{\oplus}$ denotes additive algebraic K-theory. This claim can be verified explicitly by using~\cite[Proposition 3.6]{LMT20}. 

\subsubsection*{Acknowledgements}
The authors are grateful to Akhil Mathew for discussions about and interest in this project. 
The first author would also like to thank Tyler Lawson for a number of helpful conversations.
The second author was funded by the Deutsche Forschungsgemeinschaft (DFG, German Research Foundation) under Germany's Excellence Strategy EXC 2044 390685587, Mathematics Münster: Dynamics–Geometry–Structure and the Max Planck Institute for Mathematics in Bonn while working on this project.

\section{Preliminaries on weight structures and K-theory}
The main technical apparatus for deducing our chromatic vanishing result for TR is the notion of a weight structure on a stable $\infty$-category in conjunction with the closely related theorem of the weighted heart (cf.~\cite{Fon18,HS21}). This will help us reduce to studying additive algebraic K-theory of additive $\infty$-categories. 


\begin{definition}
A weight structure on a stable $\infty$-category $\cat{C}$ consists of a pair of full subcategories $\cat{C}_{[0, \infty]}$ and $\cat{C}_{[-\infty, 0]}$ of $\cat{C}$ such that the following conditions are satisfied:
\begin{enumerate}[leftmargin=2em, topsep=5pt, itemsep=5pt]
	\item The full subcategories $\cat{C}_{[0, \infty]}$ and $\cat{C}_{[-\infty, 0]}$ are closed under retracts in $\cat{C}$. 
	\item For $X \in \cat{C}_{[-\infty, 0]}$ and $Y \in \cat{C}_{[0, \infty]}$, the mapping spectrum $\map_{\cat{C}}(X, Y)$ is connective.
	\item For every $X \in \cat{C}$, there is a fiber sequence
	\[
		X' \to X \to X''
	\]
	with $X' \in \cat{C}_{[-\infty, 0]}$ and $X''[-1] \in \cat{C}_{[0, \infty]}$. 
\end{enumerate}
The heart of the weight structure is the subcategory $\cat{C}^{\mathrm{ht}} = \cat{C}_{[0,0]}$, where $\cat{C}_{[a, b]} = \cat{C}_{[a, \infty]} \cap \cat{C}_{[-\infty, b]}$. The weight structure is said to be exhaustive if every object is bounded, in the sense that
\[
	\cat{C} = \bigcup_{n \in \Z} \cat{C}_{[-n, n]}.
\]
A weighted $\infty$-category is a stable $\infty$-category equipped with a weight structure.
\end{definition}

\begin{remark}
The heart of a weighted $\infty$-category is an additive $\infty$-category (\cite[Lemma 3.1.2]{HS21}).
\end{remark}

We recall the following terminology which will play an important role throughout this note. For every connective $\E_1$-ring $R$, let $\Proj_R^{\omega}$ denote the full subcategory of the $\infty$-category $\LMod_{R}^{\geq 0}$ spanned by those connective left $R$-modules which are finitely generated and projective. Recall that an object of $\Proj_R^{\omega}$ can be written as a retract of a finitely generated free $R$-module (cf.~\cite[Proposition 7.2.2.7]{Lur17}). For any not necessarily connective $\E_1$-ring, let $\Perf_R$ denote the $\infty$-category of perfect $R$-modules defined as the smallest stable subcategory of $\LMod_R$ which contains $R$ and is closed under retracts. The following is our main example of interest:

\begin{example} \label{example:weight_structure_perf}
For a connective $\E_1$-ring $R$, let $\Perf_{R, \geq 0}$ be the full subcategory of $\Perf_{R}$ spanned by those perfect $R$-modules which are connective, and let $\Perf_{R, \leq 0}$ denote the full subcategory of $\Perf_R$ spanned by those perfect $R$-modules $M$ which have projective amplitude $\leq 0$. This means that every $R$-linear map $M \to N$ is nullhomotopic provided that $N$ is $1$-connective. The pair $(\Perf_{R, \geq 0}, \Perf_{R, \leq 0})$ defines an exhaustive weight structure on $\Perf_R$ whose heart is equivalent to the additive $\infty$-category $\Proj_R^{\omega}$ of finitely generated projective $R$-modules (cf. \cite[1.38 \& 1.39]{Hel20}); while the proofs therein are stated for connective $\E_\infty$-rings, the same arguments work in the $\E_1$ case. 
\end{example}

The algebraic K-theory of a weighted $\infty$-category is often determined by the additive algebraic K-theory of its heart by virtue of the theorem of the weighted heart first established by Fontes~\cite{Fon18} but we also refer the reader to~\cite[Corollary 8.1.3, Remark 8.1.4]{HS21}. Let $\cat{A}$ denote an additive $\infty$-category regarded as a symmetric monoidal $\infty$-category with the cocartesian symmetric monoidal structure, so that the core $\cat{A}^{\simeq}$ inherits the structure of an $\E_\infty$-monoid. Recall that the additive algebraic K-theory of $\cat{A}$ is defined by
\[
	\K^{\oplus}(\cat{A}) = (\cat{A}^{\simeq})^{\mathrm{grp}},
\]
where $(\cat{A}^{\simeq})^{\mathrm{grp}}$ denotes the group completion of the $\E_\infty$-monoid $\cat{A}^{\simeq}$. We have the following result which will play an instrumental role below (cf.~\cite[Theorem 5.1]{Fon18} and~\cite[Corollary 8.1.3]{HS21}):

\begin{theorem} \label{theorem:weighted_heart_theorem}
The canonical map of spectra
\[
	\K^{\oplus}(\cat{C}^{\mathrm{ht}}) \to \K(\cat{C})
\]
is an equivalence for every stable $\infty$-category $\cat{C}$ equipped with an exhaustive weight structure.
\end{theorem}

\section{Chromatic vanishing results} \label{section:chromatic_vanishing_results}
The main goal of this section is to prove Theorem~\ref{theorem:A} from~\textsection\ref{section:introduction} and discuss various consequences. As explained, our proof of this result relies on the close relationship between $\TR$ and the spectrum of curves in K-theory (cf.~\cite{BS05,Hes96,McC21}). We will regard $\TR$ as a functor $\TR : \Alg_{\E_1}^{\mathrm{cn}} \to \Sp$ given by
\[
	\TR(R) \simeq \map_{\CycSp}(\widetilde{\THH}(\S[t]), \THH(R))
\]
following~\cite{McC21} and this agrees with the classical construction of $\TR$ by~\cite[Theorem 3.3.12]{McC21}. By virtue of our assumption that $R$ is connective, there is an equivalence of spectra
\[
	\TR(R) \simeq \varprojlim \Omega \tilde{\K}(R[t]/t^i),
\]
where $\tilde{\K}(R[t]/t^i)$ denotes the fiber of the map $\K(R[t]/t^i) \to \K(R)$ induced by the augmentation. In this generality, the result was obtained by the second author in~\cite{McC21} preceded by Hesselholt~\cite{Hes96} and Betley--Schlichtkrull~\cite{BS05} who proved the result for associative rings after profinite completion. With this equivalence at our disposal, we prove the following result: 

\begin{theorem} \label{theorem:vanishing_TR}
Let $n \geq 1$. If $R$ is a connective $\E_1$-ring such that $L_n^{p,f}R \simeq 0$, then
\[
	L_{T(k)} \TR(R) \simeq 0
\]
for every $1 \leq k \leq n$. 
\end{theorem}

The limit in the definition of the spectrum of curves on K-theory above does not commute with $T(k)$-localization. Instead, the proof of Theorem~\ref{theorem:vanishing_TR} relies on the following result, which is proved by combining the theorem of the weighted heart and a recent result which asserts that additive algebraic K-theory preserves infinite products of additive $\infty$-categories, due to C{\'o}rdova Fedeli~\cite{Cor23}. 

\begin{proposition} \label{proposition:product_Ktheory_vanishes}
Let $R$ be a connective $\E_1$-ring which vanishes after $L_n^{p,f}$-localization. If $\{S_i\}_{i \in I}$ is collection of connective $\E_1$-rings with a map of $\E_1$-rings $R \to S_i$ for every $i \in I$, then 
\[
	L_{T(k)} \big( \prod_{i \in I} \K(S_i) \big) \simeq 0
\]
for every $1 \leq k \leq n$.
\end{proposition}

\begin{proof}
For $i \in I$, the stable $\infty$-category $\Perf_{S_i}$ admits an exhaustive weight structure whose heart is equivalent to the additive $\infty$-category $\Proj_{S_i}^{\omega}$ by Example~\ref{example:weight_structure_perf}. The canonical composite
\[
	\K^{\oplus}\Big( \prod_{i \in I} \Proj_{S_i}^{\omega} \Big) \to \prod_{i \in I} \K^{\oplus}(\Proj_{S_i}^{\omega}) \to \prod_{i \in I} \K(\Perf_{S_i})
\]
is an equivalence by~\cite[Corollary 2.11.1]{Cor23} and Theorem~\ref{theorem:weighted_heart_theorem}, so we have reduced to proving that
\[
	L_{T(k)} \K^{\oplus}\Big( \prod_{i \in I} \Proj_{S_i}^{\omega} \Big) \simeq 0
\]
for $1 \leq k \leq n$. By~\cite[Proposition 3.6]{LMT20}, it suffices to prove that the endomorphism $\E_1$-rings of
\[
	\cat{A} = \prod_{i \in I} \Proj^{\omega}_{S_i}
\]
vanish after $L_n^{p,f}$-localization. If $P \in \cat{A}$, then the endomorphism $\E_1$-ring of $P$ is given by
\[
	\mathrm{End}_{\cat{A}}(P) \simeq \prod_{i \in I} \map_{S_i}(P_i, P_i), 
\]
where $\map_{S_i}(P_i, P_i)$ denotes the mapping spectrum in $\LMod_{S_i}$. For each $i \in I$, we may choose a positive integer $n_i \geq 1$ such that $P_i$ is a retract of $S_i^{\oplus n_i}$ by virtue of our assumption that $P_i$ is a finitely generated projective $S_i$-module. Consequently, we obtain a retract diagram of spectra
\[
	\mathrm{End}_{\cat{A}}(P) \to \prod_{i \in I} S_i^{\oplus n_i^2} \to \mathrm{End}_{\cat{A}}(P)
\]
which proves the desired statement since the middle term is a left $R$-module, hence vanishes after $L_n^{p,f}$-localization by virtue of our assumption that $R$ is $L_n^{p,f}$-acyclic.
\end{proof}

\begin{remark}
In general, $E$-acyclic spectra are not closed under infinite products; for each $n \geq 0$, the $n$th Postnikov truncation $\tau_{\leq n}\mathbb{S}$ is $K(1)$-acyclic, whereas $\prod_{n \geq 0} \tau_{\leq n}\mathbb{S}$ is not, else $L_{K(1)}\mathbb{S} \simeq 0$. 
The assumptions of Proposition~\ref{proposition:product_Ktheory_vanishes} should be viewed as a uniformity condition on the spectra $\K(S_{i})$, forcing their product to become acyclic.
\end{remark}

\begin{proof}[Proof of Theorem~\ref{theorem:vanishing_TR}]
Since $R$ is a connective $\E_1$-ring, there is an equivalence of spectra $\TR(R) \simeq \mathrm{C}(R)$ by~\cite[Corollary 4.2.5]{McC21}. Thus, the spectrum $\Sigma\TR(R)$ is the fiber of a suitable map
\[
	\prod_{i \geq 1} \widetilde{\K}(R[t]/t^i) \to \prod_{i \geq 1} \widetilde{\K}(R[t]/t^i)
\]
which proves the desired statement as these products vanish after $T(k)$-localization for $1 \leq k \leq n$ by virtue of Theorem~\ref{proposition:product_Ktheory_vanishes}.
\end{proof}

\begin{remark}
As remarked above, we have used work by C{\'o}rdova Fedeli~\cite{Cor23} in a crucial way. 
This result on K-theory of additive $\infty$-categories is part of a long tradition of examining the interaction of algebraic K-theory and infinite products of categories. 
One of the first results of this kind is due to Carlsson, who showed that K-theory preserves infinite products of exact 1-categories with a cylinder functor \cite{Carl95}. In~\cite{KW20}, Kasprowski--Winges proved that K-theory preserves infinite products of additive categories. Furthermore, Kasprowski--Winges~\cite{KW19} used a characterization of Grayson~\cite{Gra12} to prove that non-connective algebraic K-theory preserves infinite products of stable $\infty$-categories and this was used in~\cite{BKW21} with Bunke to prove the analogous statement of prestable $\infty$-categories.
\end{remark} 

\begin{remark} \label{remark:history}
Another attempt to prove Proposition~\ref{proposition:product_Ktheory_vanishes} proceeds by invoking the recent result of Kasprowski--Winges~\cite{KW19}, which asserts that the canonical map of spectra
\[
	\K\big( \prod_{i \in I} \Perf(S_i) \big) \to \prod_{i \in I} \K(S_i)
\]
is an equivalence (cf. Remark~\ref{remark:history}). Proceeding as in the proof of Proposition~\ref{proposition:product_Ktheory_vanishes}, it suffices to prove that the endomorphism $\E_1$-rings of the product of the stable $\infty$-categories $\Perf(S_i)$ vanish after $L_n^{p,f}$-localization. This is closely related to the following assertion:

\begin{enumerate}[leftmargin=2em, topsep=5pt, itemsep=5pt]
	\item[$(\ast)$] Let $E$ denote the endomorphism $\E_1$-ring of a finite spectrum $V$ of type $n$. If $v : \Sigma^k E \to E$ is the associated $v_n$ self-map of $E$, then there is a canonical lift of $v$ to a map of $E$-$E$-bimodules. 
\end{enumerate}

By the description of the $\E_1$-center as Hochschild cohomology, the statement $(\ast)$ is equivalent to asking for a lift of the class $v \in \pi_\ast(E)$ to a class $\tilde{v} \in \pi_{\ast}\cat{Z}_{\E_1}(E)$ along the $\E_1$-map $\cat{Z}_{\E_1}(E) \rightarrow E$. 
Classes which do lift in this way can be viewed as ``homotopically central" elements of $E$, and we remark that such lifts exist for all $\E_2$-rings, by the universal property of the $\E_1$-center. 

However, the assertion $(\ast)$ is false as we learned from Maxime Ramzi, and we thank him for help with the following argument. If such a lift exists, then we obtain an equivalence of $L_{K(n)}$-$L_{K(n)}$-bimodules
\[
	\varphi : \Sigma^k L_{K(n)}E \to L_{K(n)}E,
\]
and there is an equivalence of $\E_1$-rings $\mathrm{End}_{K(n)}(L_{K(n)} V) \simeq L_{K(n)}E$ since $V$ is a finite spectrum. The $\infty$-category of $K(n)$-local spectra is equivalent to the $\infty$-category $\Mod_{L_{K(n)}E}(\Sp_{K(n)})$ since $L_{K(n)}V$ is a compact generator of $\Sp_{K(n)}$. As a consequence, for every $K(n)$-local spectrum $X$, we obtain an equivalence $\Sigma^k X \to X$ by base-changing along $\varphi$. This is a contradiction since the homotopy groups of a $K(n)$-local spectrum are not periodic. We indicate an example of this at every height $n \geq 1$. Let $k$ be a perfect field of characteristic $p$, let $\mathbb{G}$ be a $1$-dimensional formal group of height $n$, and let $E_n$ denote the associated Lubin--Tate theory which canonically carries the structure of an $\E_\infty$-ring. For every topological generator $g$ of $\Z_p^{\times}$, there is a map of $\E_\infty$-rings $\psi_g : E_n \to E_n$, and we let $F_n$ denote the fiber of the map
\[
	E_n \xrightarrow{1 - \psi_g} E_n.
\]
A calculation reveals that the homotopy groups of $F_n$ are not periodic. For instance, if $n = 1$, then $F_1 \simeq L_{K(1)}\mathbb{S}$ since the map $\psi_g$ is induced by Adams operations on $E_1 \simeq \mathrm{KU}_p^{\wedge}$. 
\end{remark}

Finally, we explore some immediate consequences of Theorem~\ref{theorem:vanishing_TR}. 

\begin{corollary} \label{corollary:p_power_torsion}
Let $R$ be a connective $\E_1$-algebra over $\Z/p^j$. If $n \geq 1$, then $L_{T(n)}\TR(R) \simeq 0$.  
\end{corollary}

\begin{proof}
Note that $L_n^{p,f}R$ is a module over $L_n^{p,f}\Z/p^j \simeq 0$, so the assertion follows from Theorem~\ref{theorem:vanishing_TR}.  
\end{proof}

Recall that Corollary~\ref{corollary:p_power_torsion} above also follows from~\cite{BCM20,LMT20,Mat20} as discussed in the introduction. We deduce some consequence for connective Morava K-theory. Let $k(n)$ denote the connective cover of the $n$th Morava $\K$-theory $K(n)$. The spectrum $k(n)$ carries the structure of an $\E_1$-ring but not the structure of an $\E_2$-ring. We have the following:

\begin{corollary} \label{corollary:moravak}
If $n \geq 2$, then $L_{T(k)}\TR(k(n)) \simeq 0$ for every $1 \leq k \leq n-1$.  
\end{corollary}

\begin{proof}
For $n \geq  2$, the canonical map $k(n) \to \F_p$ is a $L_{n-1}^{p,f}$-local equivalence by~\cite[Lemma 2.2]{LMT20}, so the assertion follows from Theorem~\ref{theorem:vanishing_TR}.
\end{proof}

\begin{remark}
There is a fiber sequence of spectra
\[
	\K(\F_p) \to \K(k(n)) \to \K(K(n)),
\]
by~\cite[Proposition 4.4]{AGH19} preceded by~\cite{BL14}. We consider this as an analogue of Quillen's dévissage theorem for algebraic $\K$-theory of ring spectra. One might ask whether we can establish a similar fiber sequence for $\TR$. In particular, this would allow us to deduce an analogue of Corollary~\ref{corollary:moravak} for the non-connective Morava $K$-theory.
\end{remark}

Let $y(n)$ denote the Thom spectrum considered in \cite[Section 3]{MRS01}. This is the Thom spectrum associated to the map of $\E_{1}$-spaces
\[
\Omega J_{p^{n-1}}S^{2} \hookrightarrow \Omega^{2} S^{3} \rightarrow \mathrm{BGL}_{1}(\mathbb{S}_{p}^{\wedge})
\]
where $J_{p^{n-1}}S^{2}$ is the $2(p^{n-1})$-skeleton of $\Omega S^{3}$, which has a single cell in each even dimension. 
The map $\Omega^{2} S^{3} \rightarrow \mathrm{BGL}_{1}(\mathbb{S}^{\wedge}_{p})$ is the spherical fibration constructed by Mahowald (for $p=2$) and Hopkins (for $p$ odd) whose Thom spectrum is $H\mathbb{F}_{p}$. We have the following:

\begin{corollary} 
If $n \geq 2$, then $L_{T(k)}\TR(y(n)) \simeq 0$ for every $1 \leq k \leq n-1$. 
\end{corollary}

\begin{proof}
This follows immediately by combining Theorem~\ref{theorem:vanishing_TR} with \cite[Lemma 4.14]{LMT20}.  
\end{proof}

\begin{remark}
If $R$ is a connective $H\Z$-algebra, then the canonical map
\[
	L_{T(1)}\K(R) \to L_{T(1)}\K(R[1/p])
\]
is an equivalence by~\cite{BCM20,LMT20}. The analogue of this result does not hold for $\TC$ as explained in~\cite[Remark 4.27]{LMT20}, which in particular means that the result also does not prolong to $\TR$. However, at chromatic heights $n \geq 2$, $\TC$ does satisfy a version of chromatic purity (cf.~\cite[Corollary 4.5]{LMT20}). In particular, if $A \rightarrow B$ is an $L_{n}^{p,f}$-local equivalence of $\mathbb{E}_{1}$-rings, then the induced map
\[
L_{T(n)}\TC(\tau_{\geq 0}A) \xrightarrow{\simeq} L_{T(n)}\TC(\tau_{\geq 0}B). 
\]
is an equivalence.
One can wonder whether such a statement is true of $T(n)$-local $\TR$, but our methods here do not seem to shed light on this problem.
\end{remark}

\bibliographystyle{abbrv}
\bibliography{chromatic_vanishing_TR} 

\end{document}